\documentclass[10pt]{article}
\usepackage{mathrsfs}
\usepackage{amssymb}
\usepackage{amsmath}
\usepackage[all]{xy}
\usepackage{amsthm}
\usepackage{amsfonts}
\usepackage{latexsym}
\usepackage{amscd,amssymb,amsopn,amsmath,amsthm,graphics,amsfonts,mathrsfs,accents,enumerate,verbatim,calc}
\usepackage[dvips]{graphicx}
\usepackage[colorlinks=true,linkcolor=red,citecolor=blue]{hyperref}
\newtheorem{thm}{Theorem}[section]
\newtheorem{prop}[thm]{Proposition}

\newtheorem{lem}[thm]{Lemma}

\newtheorem{cor}[thm]{Corollary}

\newtheorem{exa}[thm]{Example}

\newtheorem{rem}[thm]{Remark}

\newtheorem{Para}[thm]{}
\newcommand{\A}{A{\text-}\Mod}

\newcommand{\p}{A{\text-}\Proj}
\newcommand{\X}{\mathscr{X}}
\newcommand{\GP}{A{\text-}\GProj}
\newcommand{\T}{\mathscr{T}}
\newcommand{\C}{\mathbb{C}}
\newcommand{\K}{\mathbb{K}}
\newcommand{\D}{\mathbb{D}}

\def\id{\mathop{\rm id}\nolimits}
\def\Im{\mathop{\rm Im}\nolimits}
\def\Ker{\mathop{\rm Ker}\nolimits}
\def\Coker{\mathop{\rm Coker}\nolimits}

\def\Mod{\mathop{\rm Mod}\nolimits}
\def\Hom{\mathop{\rm Hom}\nolimits}

\def\Ext{\mathop{\rm Ext}\nolimits}

\def\fd{\mathop{\rm fd}\nolimits}
\def\pd{\mathop{\rm pd}\nolimits}
\def\fd{\mathop{\rm fd}\nolimits}
\def\Gpd{\mathop{\rm Gpd}\nolimits}

\def\Proj{\mathop{\rm Proj}\nolimits}

\def\GProj{\mathop{\rm GProj}\nolimits}

\def\Gpd{\mathop{\rm Gpd}\nolimits}

\def\Con{\mathop{\rm Con}\nolimits}
\def\Im{\mathop{\rm Im}\nolimits}
\begin{document}

\title{\large \bf Gorenstein projective modules and recollements over triangular matrix rings
\thanks{{\it 2010 Mathematics Subject Classification}: 18E30, 16E35, 18G20.}
\thanks{{\it Keywords}: Gorenstein projective modules; triangulated categories; triangle-equivalences; recollements.}
}
\author{Huanhuan Li$^a$, Yuefei Zheng$^b$, Jiangsheng Hu$^{c}$\footnote{Corresponding author} and Haiyan Zhu$^{d}$ \\
\it\footnotesize $^a$School of Mathematics and Statistics, Xidian University, Xi'an 710071, Shaanxi Province, China\\
\it\footnotesize $^b$College of Science, Northwest A$\&$F University, Yangling 712100, Shaanxi Province, China\\
\it\footnotesize $^c$Department of Mathematics, Jiangsu University of Technology, Changzhou 213001, Jiangsu Province, China\\
\it\footnotesize $^d$College of Science, Zhejiang University of Technology, Hangzhou 310023, Zhejiang Province, China\\
%\it\small $^*$Corresponding author.\\
\it\footnotesize Email addresses: lihuanhuan0416@163.com, yuefeizheng@sina.com, jiangshenghu@jsut.edu.cn and hyzhu@zjut.edu.cn
}
\date{}
\baselineskip=14pt
\maketitle
\begin{abstract} Let $T=\left(
                          \begin{array}{cc}
                            R & M \\
                            0 & S \\
                          \end{array}
                        \right)
$ be a triangular matrix ring with $R$ and $S$ rings and $_RM_S$ an $R$-$S$-bimodule. We describe Gorenstein projective modules over $T$. In particular, we refine a result of
Enochs, Cort\'{e}s-Izurdiaga and Torrecillas [Gorenstein conditions over triangular
matrix rings, J. Pure Appl. Algebra {\bf 218} (2014), no. 8, 1544-1554]. Also, we consider when the recollement of $\D^b(T{\text-}\Mod)$ restricts to a recollement of its subcategory $\D^b(T{\text-}\Mod)_{fgp}$ consisting of complexes with finite Gorenstein projective dimension. As applications, we obtain recollements of the stable category $\underline{T{\text-}\GProj}$  and recollements of the Gorenstein defect category $\D_{def}(T{\text-}\Mod)$.
\end{abstract}

%\centerline { \bf  Abstract}
%%\bigskip
%\noindent For an abelian category $\mathscr{A}$ with enough projective objects, we prove that
%
%We study the Gorenstein projective dimensions of complexes in an abelian category \\
%\vbox to 0.3cm{}\\
%{\it Key Words:} Gorenstein projective dimensions; triangulated categories; singularity category; triangle-equivalence; recollements.\\
%{\it 2010 Mathematics Subject Classification:} 16E05; 18G20; 18G35.

\section{\bf Introduction}

The history of Gorenstein homological algebra traces back to Auslander and Bridger \cite{AB}, where they study modules of G-dimension zero over noetherian rings. This kind of modules was generalized by Enochs and Jenda \cite{EJ1}, who introduced Gorenstein projective modules over arbitrary rings. The main idea of Gorenstein homological algebra is to get a counterpart of classical results in homological algebra by replacing projective modules with Gorenstein projective modules. In view of this, how to deal with the Gorenstein version corresponding to the classical version is very important.

Recollements of triangulated categories were introduced by Beilinson, Bernstein and Deligne \cite{BBD} and play an important role in algebraic geometry and representation theory.
Certain recollements of derived module categories of rings has been considered by many authors, it is of great use in dealing with algebraic properties of one ring from another two, see for instance \cite{AKL,AKLY,CX,CL,CPS1,CPS2,Ko}.
%Meanwhile, many people consider when the recollements of derived module categories of rings restrict to that of bounded homotopy category of projective modules \cite{AKLY,Ko,LL}.
Let $T=\left(
                                                                                 \begin{array}{cc}
                                                                                   R & M \\
                                                                                   0 & S \\
                                                                                 \end{array}
                                                                               \right)$ be the triangular matrix ring with $R,S$ rings and $M$ an $R$-$S$-bimodule.  It is known that if $\pd M_S<\infty$ then we have the following recollement of bounded derived category  $\D^b(T{\text-}\Mod)$ of $T$-modules:
$$\xymatrix{\D^b( R{\text-}\Mod)\ar[r]^{\D^b(i_*)}&\D^b( T{\text-}\Mod)\ar@/^1pc/[l]^{\D^b(i^!)}\ar@/_1pc/[l]_{\mathbb{L}^b(i^*)}\ar[r]^{\D^b(j^*)} &\D^b( S{\text-}\Mod)\ar@/^1pc/[l]^{\D^b(j_*)}\ar@/_1pc/[l]_{\mathbb{L}^b(j_!)}, }\eqno{(1.1)}$$
where these six functors are the derived versions of those as defined in Lemma \ref{lem:5.1} (1).
Liu and Lu \cite{LL} show that if $\pd_RM<\infty$, then (1.1) restricts to the following recollement of homotopy category $\K^b( T{\text-}\Proj)$ of
$T$-projective modules (see Lemma \ref{lem:5.2}):
$$\xymatrix{\K^b( R{\text-}\Proj)\ar[r]^{\D^b(i_*)}&\K^b( T{\text-}\Proj)\ar@/^1pc/[l]^{\D^b(i^!)}\ar@/_1pc/[l]_{\mathbb{L}^b(i^*)}\ar[r]^{\D^b(j^*)} &\K^b( S{\text-}\Proj)\ar@/^1pc/[l]^{\D^b(j_*)}\ar@/_1pc/[l]_{\mathbb{L}^b(j_!)}}.\eqno{(1.2)}$$
Note that $\K^b( T{\text-}\Proj)$ is the smallest triangulated subcategory of $\D^b(T{\text-}\Mod)$ containing the class $T{\text-}\Proj$ of $T$-projective modules. Denote by $\langle T{\text-}\GProj\rangle$ the smallest triangulated subcategory of $\D^b(T{\text-}\Mod)$ containing the class $T{\text-}\GProj$ of $T$-Gorenstein projective modules. Naturally, we have the following question:

{\bf Question.} Whether and when the recollement (1.1) of $\D^b(T{\text-}\Mod)$ restricts the following recollement of $\langle T{\text-}\GProj\rangle$:
$$\xymatrix{\langle R{\text-}\GProj\rangle\ar[r]&\langle T{\text-}\GProj\rangle\ar@/^1pc/[l]\ar@/_1pc/[l]\ar[r] &\langle S{\text-}\GProj\rangle\ar@/^1pc/[l]\ar@/_1pc/[l]}?$$

In order to solve this question, we first give an explicit description for Gorenstein projective $T$-modules over the triangular matrix ring $T$.  We get the following:

\begin{thm}\label{thm:1.1} Let $T=\left(
                                                                                 \begin{array}{cc}
                                                                                   R & M \\
                                                                                   0 & S \\
                                                                                 \end{array}
                                                                               \right)
$ be a triangular matrix ring with $\pd_RM<\infty$ and $\fd M_S<\infty$.
Then
$\left(
                                                                                        \begin{array}{c}
                                                                                          X \\
                                                                                          Y \\
                                                                                        \end{array}
                                                                                      \right)_\phi$ is a Gorenstein projective left $T$-module if and only if $Y$  is a Gorenstein projective left $S$-module and $\phi:M\otimes_S Y\to X$ is an injective $R$-morphism with a Gorenstein projective cokernel.
\end{thm}
We note that the above theorem generalizes Theorem 3.5 in \cite{ECT} to a more general case, whereas the authors therein assume the ring  $R$ to be left Gorenstein regular.

Meanwhile, for an arbitrary ring $A$, we would like to know the exact form of objects in $\langle A{\text-}\GProj\rangle$. Inspired by \cite{Vel}, an object $X^\bullet\in\D^b(\A)$ is said to have {\it finite Gorenstein projective dimension} if $X^\bullet$ is isomorphic to some bounded complex consisting of Gorenstein projective modules. Denote by $\D^b(\A)_{fgp}$ the subcategory of $\D^b(\A)$ consisting of complexes with finite Gorenstein projective dimension. It is shown in the appendix that $\D^b(\A)_{fgp}$ is a thick subcategory of $\D^b(\A)$ with $\D^b(\A)_{fgp}=\langle A{\text-}\GProj\rangle$ (see Theorem \ref{thm:A.3}). Besides, we get two equivalences of triangulated categories:
$\underline{\GP}\simeq\D^b(\A)_{fgp}/\K^b(\p)$ and $\D_{def}(\A)\simeq\D^b(\A)/\D^b(\A)_{fgp}$ (see Theorem \ref{thm:A.5}).

According to the previous results, we get an answer to the above question.
\begin{thm}\label{thm:1.3} Let $T=\left(
                                                                                 \begin{array}{cc}
                                                                                   R & M \\
                                                                                   0 & S \\
                                                                                 \end{array}
                                                                               \right)
$ be a triangular matrix ring with $\pd_RM<\infty$ and $\pd M_S<\infty$. Denote by $S{\text-}\GProj$ the class of Gorenstein projective left $S$-modules. Then
\begin{enumerate}
\item[(1)] (1.1) restricts to the following left recollement:
$$\xymatrix{\D^b( R{\text-}\Mod)_{fgp}\ar[r]^{\D^b(i_*)}&\D^b( T{\text-}\Mod)_{fgp}\ar@/_1pc/[l]_{\mathbb{L}^b(i^*)}\ar[r]^{\D^b(j^*)} &\D^b( S{\text-}\Mod)_{fgp}\ar@/_1pc/[l]_{\mathbb{L}^b(j_!)}}.\eqno{(1.3)}$$
\item[(2)] Assume further $\Gpd {_R}M\otimes_S G<\infty$ for any $G\in S{\text-}\GProj$. Then (1.1) restricts the following recollement:
$$\xymatrix{\D^b( R{\text-}\Mod)_{fgp}\ar[r]^{\D^b(i_*)}&\D^b( T{\text-}\Mod)_{fgp}\ar@/^1pc/[l]^{\D^b(i^!)}\ar@/_1pc/[l]_{\mathbb{L}^b(i^*)}\ar[r]^{\D^b(j^*)} &\D^b( S{\text-}\Mod)_{fgp}\ar@/^1pc/[l]^{\D^b(j_*)}\ar@/_1pc/[l]_{\mathbb{L}^b(j_!)}}.\eqno{(1.4)}$$
\end{enumerate}
\end{thm}
We note that the condition ``$\Gpd {_R}M\otimes_S G<\infty$ for any $G\in S{\text-}\GProj$'' in Theorem \ref{thm:1.3} (2) arises naturally in representation theory (see Example \ref{exa:5.5}). As applications of Theorem \ref{thm:1.3}, we obtain (left) recollement of stable category $\underline{T{\text-}\GProj}$ of Gorenstein projective left $T$-modules and (left) recollement of Gorenstein defect category $\D_{def}(T{\text-}\Mod)$ of $T$-modules (see Corollaries \ref{cor:5.3} and \ref{cor:5.4}). These are generalizations of corresponding results of \cite{ZP} and \cite{L}.

The contents of this paper are outlined as follows. In Section \ref{pre}, we fix notations and collect some basic facts needed in our later proofs.
In Section \ref{Gor}, we study Gorenstein projective modules over triangular matrix rings and prove Theorem \ref{thm:1.1}. In Section \ref{singcat}, we consider when the recollement of $\D^b(T{\text-}\Mod)$ restricts to a recollement of its subcategory $\D^b(T{\text-}\Mod)_{fgp}$ consisting of complexes with finite Gorenstein projective dimension, including the proof of Theorem \ref{thm:1.3}.

\section{\bf Preliminaries}\label{pre}

\vspace{0.2cm}
Throughout this paper, all rings will be associative with identity and all subcategories will be full additive and closed under isomorphisms. For any ring $A$, the category of left $A$-modules is denoted by $A{\text-}\Mod$. Unless otherwise specified we will be working with left modules. For any $M\in A{\text-}\Mod$, we use $1_M$ (or $1$ for short) to denote the identity map of $M$.
We use $A{\text-}\Proj$  to denote the subcategories of left projective $A$-modules. Usually, we use $_AM$ (resp. $M_A$) to denote a left (resp. right) $A$-module $M$, and the  projective, injective and flat dimensions of $_AM$ (resp. $M_A$) will be denoted by $\pd_AM$, $\id_AM$ and $\fd_AM$ (resp. $\pd M_A$, $\id M_A$ and $\fd M_A$) respectively. For a subclass $\X$ of left $A$-modules. Denote by $\X^\bot$ (resp. $^\bot\X$) the subcategory consisting of modules $M\in A{\text-}\Mod$ such that $\Ext_A^1(X,M)=0$ (resp. $\Ext_A^1(M,X)=0$) for any $X\in\X$.

For an additive category $\X$, denote by $\C(\X)$ the category of $\X$-complexes; the objects are complexes of objects in $\X$ and morphisms are chain maps. As usual, by adding a superscript  $*\in\{+,\ -, \ b\}$, we denote their corresponding $*$-bounded categories. For example, $\C^-(\X)$ is the category of bounded above complexes. Usually, an object of $\C^*(\X)$ is written as $X^\bullet=(X^i,d_X^i)_{i\in\mathbb{Z}}$, where $d_X^i:X^i\to X^{i+1}$ is the $i$th differential of $X^\bullet$. For the sake of simplicity, we also write $X^\bullet=(X^i)_{i\in\mathbb{Z}}$ for short. If $\X$ is an abelian category, for any $X^\bullet\in\C^*(\X)$ and any integer $n$. We set $Z^n(X^\bullet)=\Ker d_X^n$ and $B^n(X^\bullet)=\Im d_X^{n-1}$. Denote by $H^n(X^\bullet)=Z^n(X^\bullet)/B^n(X^\bullet)$ and by $H(X^\bullet)$ the homology complex of $X^\bullet$. $X^{\bullet}$ is called  acyclic (or exact) if $H^{n}(X^{\bullet})=0$ for any $n\in \mathbb{Z}$. Given a chain map $f:X^\bullet\to Y^\bullet$, we use $\Con(f)$ to denote the mapping cone of $f$. A chain map $f:X^\bullet\to Y^\bullet$ is called a quasi-isomorphism if $H(f):H(X^\bullet)\to H(Y^\bullet)$ is an isomorphism, or equivalently, $\Con(f)$ is acyclic.

Let $A$ be a ring and $X^\bullet\in\C(A{\text-}\Mod)$. Recall that $X^\bullet$ is called complete projective if it is acyclic with each $X^i$ projective and the complex $\Hom_{A}(X^\bullet,P)$ remains acyclic for any projective $A$-module $P$. A module $M$ is called  Gorenstein projective if $M\cong Z^0(X^\bullet)$ for some complete projective complex $X^\bullet$.
The class of all Gorenstein projective $A$-modules is denoted by $A{\text-}\GProj$.
The Gorenstein projective dimension $\Gpd_AM$ of the module $M$ is defined by declaring that $\Gpd_AM\leq n$ if, and only if there is an exact sequence $0\to G_n\to \cdots\to G_1\to G_0\to M\to0$ with all $G_i$ Gorenstein projective for some nonnegative integer $n$ (see \cite[Definition2.8]{Ho}).
We refer to \cite{AB,AM,CFH,EJ1,EJ2,Ho} for more details.
%
%Let $A$ be a ring. It is well known that every projective $A$-module is flat. However it is not necessarily true in the Gorenstein context, that is there exists Gorenstein projective module which is not Gorenstein flat, see for example [?].

\begin{lem}\label{lem:2.1}  The following hold true for a ring $A$.

(1) Let $X^\bullet$ be an acyclic complex with each $X^i\in A{\text-}\Proj$ and $M\in A{\text-}\Mod$. Then $\Hom_A(X^\bullet,M)$ is acyclic if $\id_AM<\infty$. Furthermore, assume $X^\bullet$ is complete projective, then $\Hom_A(X^\bullet,M)$ is acyclic if $\pd_AM<\infty$.

(2) Let $X^\bullet$ be an acyclic complex with each $X^i\in A{\text-}\Proj$ and $M\in\Mod A$. Then $M\otimes_A X^\bullet$ is acyclic if $\fd M_A<\infty$.
\end{lem}

\begin{proof} We only prove (1) and the proof of (2) is similar.
Let $X^\bullet$ be an acyclic complex with each $X^i\in A{\text-}\Proj$ and $\id_AM<\infty$. Then there exists an exact sequence $0\to M\to E^0\to E^1\to \cdots\to E^n\to0$ with each $E^i$ injective for some integer $n\geq0$. Since each $X^i\in A{\text-}\Proj$, we get an exact sequence of complexes of abelian groups $0\to\Hom_A(X^\bullet,M)\to\Hom_A(X^\bullet,E^0)\to\Hom_A(X^\bullet,E^1)\to\cdots\to\Hom_A(X^\bullet,E^n)\to0$.
Note that each $\Hom_A(X^\bullet,E^i)$ is acyclic. Then $\Hom_A(X^\bullet,M)$ is acyclic. Now assume further that $X^\bullet$ is complete projective. Note that $\Hom_A(X^\bullet,P)$ is acyclic for any projective $A$-module $P$. So we deduce $\Hom_A(X^\bullet,M)$ is acyclic if $\pd_AM<\infty$ in a similar way by choosing a projective resolution of $M$.
\end{proof}

Let $R$ and $S$ be two rings, ${_R}M_S$ an $R$-$S$-bimodule, and $T=\left(
                                                                                 \begin{array}{cc}
                                                                                   R & M \\
                                                                                   0 & S \\
                                                                                 \end{array}
                                                                               \right)
$ the triangular matrix ring. A left $T$-module is identified with a triple $\left(
                                                                                        \begin{array}{c}
                                                                                          X \\
                                                                                          Y \\
                                                                                        \end{array}
                                                                                      \right)_\phi$,
 where $X\in R{\text-}\Mod$, $Y\in S{\text-}\Mod$ and $\phi:M\otimes_S Y\to X$ ia an $R$-morphism. If there is no possible confusion, we shall omit the morphism $\phi$ and write $\left(
                                                                                        \begin{array}{c}
                                                                                          X \\
                                                                                          Y \\
                                                                                        \end{array}
                                                                                      \right)$ for short. For example, we write $\left(
                                                                                        \begin{array}{c}
                                                                                          M\otimes_SY \\
                                                                                          Y \\
                                                                                        \end{array}
                                                                                      \right)$ for the $T$-module $\left(
                                                                                        \begin{array}{c}
                                                                                          M\otimes_SY \\
                                                                                          Y \\
                                                                                        \end{array}
                                                                                      \right)_1$.
%Given an $R$-map $\phi:M\otimes_SY\to X$, we denote $\widetilde{\phi}:Y\to \Hom_R(M,X)$
A $T$-morphism
$\left(
\begin{array}{c}
 X \\
  Y \\
\end{array}
\right)_\phi\to \left(
\begin{array}{c}
X' \\
Y' \\
\end{array}
\right)_{\phi'}$
will be identified with a pair
$\left(
\begin{array}{c}
f \\
g \\
\end{array}
\right)$,
where $f\in\Hom_R(X,X')$ and $g\in\Hom_S(Y,Y')$, such that the following diagram commutes:
$$\xymatrix{M\otimes_S Y\ar[r]^\phi\ar[d]_{1\otimes g} & X\ar[d]_f\\
M\otimes_S Y'\ar[r]^{\phi'} &X'.
}$$

A sequence $0\to\left(
\begin{array}{c}
 X_1 \\
 Y_1 \\
\end{array}
\right)_{\phi_1}\xrightarrow{\left(
\begin{array}{c}
f_1 \\
g_1 \\
\end{array}
\right)}\left(
\begin{array}{c}
 X_2 \\
 Y_2 \\
\end{array}
\right)_{\phi_2}\xrightarrow{\left(
\begin{array}{c}
f_2 \\
g_2 \\
\end{array}
\right)}\left(
\begin{array}{c}
 X_3 \\
 Y_3 \\
\end{array}
\right)_{\phi_3}\to0$
in $ T{\text-}\Mod$ is exact if and only if $0\to X_1\xrightarrow{f_1}X_2\xrightarrow{f_2}X_3\to0$ and $0\to Y_1\xrightarrow{g_1}Y_2\xrightarrow{g_2}Y_3\to0$ are exact in $R{\text-}\Mod$ and $S{\text-}\Mod $ respectively.
%Indecomposable projective left $T$-modules are exactly $\left(
%\begin{array}{c}
% P \\
% 0 \\
%\end{array}
%\right)_{0}$
%and
%$\left(
%\begin{array}{c}
% M\otimes_B Q \\
% Q \\
%\end{array}
%\right)_{id}$,
%where $P$ and $Q$ runs over indecomposable projective $A$-modules and $B$-modules respectively.

Dually, a right $T$-module is identified with a triple $\left(
                                                               \begin{array}{cc}
                                                                 X & Y \\
                                                               \end{array}
                                                             \right)_\psi$,
where $X\in\Mod{\text-}R$, $Y\in\Mod{\text-}S$ and $\psi:X\otimes_R M\to Y$ ia a $S$-morphism.
We refer the reader to \cite{ARS,G} for more details.

\begin{lem}\label{lem:2.2}\cite[Theorem 3.1]{HV2} Let $T=\left(
                                                                                 \begin{array}{cc}
                                                                                   R & M \\
                                                                                   0 & S \\
                                                                                 \end{array}
                                                                               \right)
$ be a triangular matrix ring. Then
$\left(
                                                                                        \begin{array}{cc}
                                                                                          X \\
                                                                                          Y \\
                                                                                        \end{array}
                                                                                      \right)_\phi$ is a projective $T$-module if and only if $Y$ is a projective $S$-module and $\phi:M\otimes_S Y\to X$ is an injective $R$-morphism with a projective cokernel.
\end{lem}

\begin{lem}\label{lem:2.3} Let $T=\left(
                                                                                 \begin{array}{cc}
                                                                                   R & M \\
                                                                                   0 & S \\
                                                                                 \end{array}
                                                                               \right)
$ be a triangular matrix ring.
\begin{enumerate}
\item[(1)] For any $X\in R{\text-}\Mod$, one has $\pd _RX=\pd_T\left(\begin{array}{c}X \\ 0 \\\end{array}\right)$.
\item[(2)] Suppose $\pd_RM<\infty$ and $Y\in S{\text-}\Mod$. Then $\pd_SY<\infty$ if and only if $\pd_T\left(\begin{array}{c}0 \\ Y \\\end{array}\right)<\infty$.
\end{enumerate}
\end{lem}
\begin{proof} (1) could be easily checked by choosing a projective resolution (see also \cite[Lemma 3.1]{C1}). For (2), suppose $\pd_T\left(\begin{array}{c}0 \\ Y \\\end{array}\right)=n<\infty$.
Then there exists a projective resolution of $\left(\begin{array}{c}0 \\ Y \\\end{array}\right)$ $$0\to\left(
                                                 \begin{array}{c}
                                                   P_n \\
                                                   Q_n \\
                                                 \end{array}
                                               \right)_{\phi_n}\to\left(
                                                 \begin{array}{c}
                                                   P_{n-1} \\
                                                   Q_{n-1} \\
                                                 \end{array}
                                               \right)_{\phi_{n-1}}\to\cdots\to\left(
                                                 \begin{array}{c}
                                                   P_1 \\
                                                   Q_1 \\
                                                 \end{array}
                                               \right)_{\phi_1}\to\left(
                                                 \begin{array}{c}
                                                   P_0 \\
                                                   Q_0 \\
                                                 \end{array}
                                               \right)_{\phi_0}\to\left(\begin{array}{c}0 \\ Y \\\end{array}\right)\to0.$$ Hence we get an exact sequence of left $S$-modules
                                               $$0\to Q_n\to Q_{n-1}\to\cdots\to Q_1\to Q_0\to Y\to0.$$
                                               Since each $\left(
                                                 \begin{array}{c}
                                                   P_i \\
                                                   Q_i \\
                                                 \end{array}
                                               \right)_{\phi_i}\in T{\text-}\Proj$, $Q_i\in S{\text-}\Proj$. Thus the above exact sequence is a projective resolution of $Y$, and hence $\pd_SY\leq n$.
                                               Conversely, assume that $\pd_S Y=n<\infty$, we will proceed by induction on $n$. If $n=0$, then $Y$ is projective.
Consider the following exact sequence of left $T$-modules
$$0\to \left(\begin{array}{c} M\otimes_S Y \\ 0 \\\end{array}\right)\to\left(\begin{array}{c} M\otimes_S Y \\ Y \\\end{array}\right)\to\left(\begin{array}{c} 0 \\ Y \\\end{array}\right)\to0.$$
Since $\pd_RM<\infty$ and $Y\in S{\text-}\Proj$, this implies $\pd_RM\otimes_S Y<\infty$ and then $\pd_T\left(\begin{array}{c} M\otimes_S Y \\ 0 \\\end{array}\right)<\infty$ by (1). Note that
$\left(\begin{array}{c} M\otimes_S Y \\ Y \\\end{array}\right)\in T{\text-}\Proj$. Then $\pd_T\left(\begin{array}{c}0 \\ Y \\\end{array}\right)<\infty$. Now suppose $\pd_SY=n\geq1$ and the assertion holds true for any interger less than $n$. Then we have an exact sequence of $S$-modules $0\to K\to Q\to Y\to0$ with $Q\in S{\text-}\Proj$ and $\pd_SK=n-1$. Hence the following is an exact sequence of left $T$-modules:
$$0\to\left(
        \begin{array}{c}
          0 \\
          K \\
        \end{array}
      \right)\to\left(
        \begin{array}{c}
          0 \\
          Q \\
        \end{array}
      \right)\to\left(
        \begin{array}{c}
          0 \\
          Y \\
        \end{array}
      \right)\to0
.$$  By the induction hypothesis, both $\left(
        \begin{array}{c}
          0 \\
          K \\
        \end{array}
      \right)$ and $\left(
        \begin{array}{c}
          0 \\
          Q \\
        \end{array}
      \right)$ are of finite projective dimensions, so is $\left(
        \begin{array}{c}
          0 \\
          Y\\
        \end{array}
      \right)$.
\end{proof}

Denote by $e_S=\left(
     \begin{array}{cc}
       0& 0 \\
       0 & 1 \\
     \end{array}
   \right)$
the idempotent element of $T$. Note that $R\cong T/Te_ST$ as rings. Thus, every left or right $R$-module has a natural $T$-module structure. For example, $_TR$ is identified with $\left(
                    \begin{array}{c}
                      R \\0
                    \end{array}
                  \right)$ and $R_T$ is identified with $\left(
                    \begin{array}{cc}
                      R & 0
                    \end{array}
                  \right)$. On the other hand, $e_R=\left(
     \begin{array}{cc}
       1& 0 \\
       0 & 0 \\
     \end{array}
   \right)$ is another idempotent element of $T$. We have the ring homomorphism $S\cong T/Te_RT$, so every left or right $S$-module has a natural $T$-module structure in a similar way.

\vspace{0.3cm}
We have the following facts.

\begin{lem}\label{lem:2.5} Let $R$ and $S$ be two rings, ${_R}M_S$ an $R$-$S$-bimodule, and $T=\left(
                                                                                 \begin{array}{cc}
                                                                                   R & M \\
                                                                                   0 & S \\
                                                                                 \end{array}
                                                                               \right)
$ the triangular matrix ring. The following statements hold true:

(1) $_RM$ is of finite projective dimension if and only if $_TS$ is of finite projective dimension.

(2) $M_S$ is of finite projective  dimension if and only if $R_T$ is of finite projective  dimension.
\end{lem}

\begin{proof} We only prove (1) and the proof of (2) is similar. Actually, consider the following exact sequence of left $T$-modules
$$0\to\left(
        \begin{array}{cc}
          R & M \\
          0 & 0 \\
        \end{array}
      \right)\to T\to {_TS}\to0.$$
It is easy to verify $\left(
        \begin{array}{cc}
          R & M \\
          0 & 0 \\
        \end{array}
      \right)\cong\left(
                         \begin{array}{c}
                           R \\ 0
                         \end{array}
                       \right)\oplus\left(
                         \begin{array}{c}
                           M \\ 0
                         \end{array}
                       \right)$.
We get that $\left(
                         \begin{array}{c}
                           M \\ 0
                         \end{array}
                       \right)$ is of finite projective dimension if and only if $_TS$ is of finite projective dimension, since $T$ and $\left(
                         \begin{array}{c}
                           R\\ 0
                         \end{array}
                       \right)$ are projective left $T$-modules.
By Lemma \ref{lem:2.3}, $\pd_T\left(
                         \begin{array}{c}
                           M \\ 0
                         \end{array}
                       \right)=\pd_RM$. Then this assertion follows.
\end{proof}

Recall from \cite{BBD} that a {\it recollement of triangulated categories} is a diagram of triangulated categories and triangle-functors
$$\xymatrix{\T'\ar[r]^{i_*}&\T\ar@/^1pc/[l]^{i^!}\ar@/_1pc/[l]_{i^*}\ar[r]^{j^*} &\T''\ar@/^1pc/[l]^{j_*}\ar@/_1pc/[l]_{j_!} }\eqno{(2.1)}$$
satisfying:

(R1) $(i^*,i_*)$, $(i_*,i^!)$, $(j_!,j^*)$ and $(j^*,j_*)$ are adjoint pairs;

(R2) $i_*$, $j_!$ and $j_*$ are fully faithful;

(R3) $j^*i_*=0$;

(R4) For each $X\in\T$, there are two triangles in $\T$ induced by counit and unit adjunctions:

$$i_*i^!X\to X\to j_*j^*X\to i_*i^!X[1] \ \textrm{and} \ j_!j^*X\to X\to i_*i^*X\to j_!j^*X[1],$$

 where $[1]$ is the shift functor of $\T$.

\vspace{2mm}
By a {\it left recollement of triangulated categories}, we mean a diagram of triangulated categories and triangle-functors consisting of the upper two rows of (2.1), satisfying conditions (R1)-(R4) which involve only the functors
$i^*,i_*,j_!$ and $j^*$.

We call diagram (2.1) a {\it recollement of abelian categories}, if $\T, \T', \T''$ in diagram (2.1) are abelian categories, the six functors involving are additive functors and conditions (R1), (R2) and (R5) are satisfied, where

(R5) $\Im i_*=\Ker j^*.$

\vspace{2mm}
Let $A$ be a ring and $e\in A$ be an idempotent. Then the natural ring homomorphism $\pi:A\to A/AeA$ is an epimorphism (in the category of rings), and hence we have a fully faithful exact functor, denoted by $i_*:A/AeA{\text-}\Mod\to A{\text-}\Mod$. Moreover, $i_*$ has a left adjoint $i^*=A/AeA\otimes_A-$ and a right adjoint $i^!=\Hom_A(A/AeA,-)$. On the other hand, if we consider the {\it restriction functor} (see \cite[I.6]{ASS}) $j^*: A{\text-}\Mod\to eAe{\text-}\Mod$ defined by $j^*(M)=eM$, $j^*$ also has a left adjoint $j_!=Ae\otimes_{eAe}-$ and a right adjoint $j_*=\Hom_{eAe}(eA,-)$. We have the following.

\begin{lem}\label{lem:2.6} Let $A$ be a ring and $e\in A$ be an idempotent.
\begin{enumerate}
\item[(1)] (\cite{P}) We have the following recollement of module categories:
$$\xymatrix{A/AeA{\text-}\Mod\ar[r]^{i_*}& A{\text-}\Mod\ar@/^1pc/[l]^{i^!}\ar@/_1pc/[l]_{i^*}\ar[r]^{j^*} & eAe{\text-}\Mod\ar@/^1pc/[l]^{j_*}\ar@/_1pc/[l]_{j_!}.}$$
\item[(2)]  (\cite{CPS1,CPS2,PS})  Assume that
(i) $\Ext_A^n(A/AeA,P)=0$ for every left projective $A$-module $P$ and every integer $n>0$;
(ii) $A/AeA$ has finite projective dimensions both as a left and right $A$-module;
(iii) $\pd_{eAe}eA<\infty$;
(iv) $\pd Ae_{eAe}<\infty$.
Then there exists a recollement of the bounded derived categories:
$$\xymatrix{\D^b( A/AeA{\text-}\Mod)\ar[r]^{\D^b(i_*)}&\D^b( A{\text-}\Mod)\ar@/^1pc/[l]^{\mathbb{R}^b(i^!)}\ar@/_1pc/[l]_{\mathbb{L}^b(i^*)}\ar[r]^{\D^b(j^*)} &\D^b(eAe{\text-}\Mod )\ar@/^1pc/[l]^{\mathbb{R}^b(j_*)}\ar@/_1pc/[l]_{\mathbb{L}^b(j_!)}, }$$
where these six functors are the derived versions of those mentioned above.
\end{enumerate}
\end{lem}

\section{\bf Gorenstein projective modules over triangular matrix rings}\label{Gor}

In this section, let $R$ and $S$ be two rings, ${_R}M_S$ an $R$-$S$-bimodule. Denote by $T=\left(
                                                                                 \begin{array}{cc}
                                                                                   R & M \\
                                                                                   0 & S \\
                                                                                 \end{array}
                                                                               \right)$ the triangular matrix ring. When $T$ is an Artin algebra, the class of finitely generated Gorenstein projective modules over $T$ is explicitly described by Zhang \cite{ZP}. Then Enochs, Cort\'{e}s-Izurdiaga and Torrecillas characterized in \cite{ECT} when an arbitrary left $T$-module is Gorenstein projective in the case that $R$ is left Gorenstein regular.
%Inspired by these, Zhu, Liu and Wang considered in \cite{ZLW} the Gorenstein homological dimensions for an arbitrary left $T$-module when $T$ is left Gorenstein regular.
In this section, we will study when an arbitrary left $T$-module is Gorenstein projective.

\begin{lem}\label{lem:3.3} Let $T=\left(
                                                                                 \begin{array}{cc}
                                                                                   R & M \\
                                                                                   0 & S \\
                                                                                 \end{array}
                                                                               \right)
$ be a triangular matrix ring with $\pd_RM<\infty$ and $\fd M_S<\infty$. The following statements hold true:

(1) If $Y\in S{\text-}\GProj$, then $\left(
                                                                                        \begin{array}{c}
                                                                                          M\otimes_SY \\
                                                                                          Y \\
                                                                                        \end{array}
                                                                                      \right)\in T{\text-}\GProj$;

(2) If $X\in R{\text-}\GProj$, then $\left(
                                                                                        \begin{array}{c}
                                                                                          X \\
                                                                                          0 \\
                                                                                        \end{array}
                                                                                      \right)\in T{\text-}\GProj$.
\end{lem}

\begin{proof} (1) Assume that $Y\in S{\text-}\GProj$. Then there exists a complete projective $S$-complex
$$G^\bullet=\cdots\to G^{-1}\xrightarrow{d^{-1}}G^0\xrightarrow{d^{0}}G^1\to\cdots$$
such that $Y\cong Z^0(G^\bullet)$. Since $\fd M_S<\infty$, following Lemma \ref{lem:2.1}, $M\otimes_S G^\bullet$ is an acyclic complex with $M\otimes_SY\cong Z^0(M\otimes_S G^\bullet)$. Now we get an acyclic $T$-complex
$$\widetilde{G^\bullet}=\cdots\to\left(\begin{array}{c}M\otimes_SG^{-1} \\G^{-1} \\\end{array}\right)\xrightarrow{\left(\begin{array}{c}1\otimes d^{-1} \\d^{-1} \\\end{array}\right)}\left(\begin{array}{c}M\otimes_SG^{0} \\G^{0} \\\end{array}\right)\xrightarrow{\left(\begin{array}{c}1\otimes d^{0} \\d^{0} \\\end{array}\right)}\left(\begin{array}{c}M\otimes_SG^{1} \\G^{1} \\\end{array}\right)\to\cdots$$ with $\left(
                                                                                        \begin{array}{c}
                                                                                          M\otimes_SY \\
                                                                                          Y \\
                                                                                        \end{array}
                                                                                      \right)\cong Z^0(\widetilde{G^\bullet})$.
It follows from Lemma \ref{lem:2.2} that each $\left(\begin{array}{c}M\otimes_SG^{i} \\G^{i} \\\end{array}\right)\in T{\text-}\Proj$.
Let $\left(\begin{array}{c}P \\Q \\\end{array}\right)_{\phi}\in T{\text-}\Proj$. Since $Q\in S{\text-}\Proj$, $\Hom_T(\widetilde{G^\bullet},\left(\begin{array}{c}P \\Q \\\end{array}\right)_{\phi})\cong\Hom_S(G^\bullet,Q)$ is acyclic. Then $\widetilde{G^\bullet}$ is a complete projective $T$-complex and hence $\left(
                                                                                        \begin{array}{c}
                                                                                          M\otimes_SY \\
                                                                                          Y \\
                                                                                        \end{array}
                                                                                      \right)\in T{\text-}\GProj$.

(2) Assume that $X\in R{\text-}\GProj$. Then there exists a complete projective $R$-complex
$$F^\bullet=\cdots\to F^{-1}\xrightarrow{d^{-1}}F^0\xrightarrow{d^{0}}F^1\to\cdots$$
such that $X\cong Z^0(F^\bullet)$. Consider the following acyclic complex
$$\widetilde{F^\bullet}=\cdots\to\left(\begin{array}{c}F^{-1} \\0 \\\end{array}\right)\xrightarrow{\left(\begin{array}{c} d^{-1} \\0 \\\end{array}\right)}\left(\begin{array}{c}F^{0} \\0 \\\end{array}\right)\xrightarrow{\left(\begin{array}{c}d^{0} \\0 \\\end{array}\right)}\left(\begin{array}{c}F^{1} \\0 \\\end{array}\right)\to\cdots.$$
Note that each $\left(\begin{array}{c}F^{i} \\0 \\\end{array}\right)\in T{\text-}\Proj$ and $\left(\begin{array}{c}X \\0 \\\end{array}\right)\cong Z^0(\widetilde{F^\bullet})$. It sufficies to show $\Hom_T(\widetilde{F^\bullet},\left(\begin{array}{c}P \\Q \\\end{array}\right)_{\phi})$ is acyclic for any $\left(\begin{array}{c}P \\Q \\\end{array}\right)_{\phi}\in T{\text-}\Proj$. Since $\Hom_T(\widetilde{F^\bullet},\left(\begin{array}{c}P \\Q \\\end{array}\right)_{\phi})\cong\Hom_R(F^\bullet,P)$, we will show $\Hom_R(F^\bullet,P)$ is acyclic.
As $\left(\begin{array}{c}P \\Q \\\end{array}\right)_{\phi}\in T{\text-}\Proj$, we have an exact sequence of left $R$-modules $0\to M\otimes_SQ\xrightarrow{\phi}P\to\Coker\phi\to0$. Then we get an exact sequence of complexes
$0\to\Hom_R(F^\bullet,M\otimes_SQ)\to\Hom_R(F^\bullet,P)\to\Hom_R(F^\bullet,\Coker\phi)\to0$. Since $\pd_RM<\infty$ and $Q\in S{\text-}\Proj$, we have $\pd_RM\otimes_SQ<\infty$. Besides, one has $\Coker\phi\in R{\text-}\Proj$. Thus both $\Hom_R(F^\bullet,M\otimes_SQ)$ and $\Hom_R(F^\bullet,\Coker\phi)$ are acyclic by Lemma \ref{lem:2.1}, and then $\Hom_R(F^\bullet,P)$ is acyclic. This completes the proof.
\end{proof}

{\bf Proof of Theorem \ref{thm:1.1}.}  For the ``if" part. Let $\phi:M\otimes_S Y\to X$ be an injective $R$-morphism. Then we get the following exact sequence of left $T$-modules
                                                                                      $$0\to\left(\begin{array}{c}M\otimes_SY \\Y \\\end{array}\right)\xrightarrow{\left(\begin{array}{c}\phi \\1_Y \\\end{array}\right)}\left(\begin{array}{c}X \\Y \\\end{array}\right)_\phi\to\left(\begin{array}{c}\Coker\phi  \\0 \\\end{array}\right)\to0.$$
Since $Y\in S{\text-}\GProj$ and $\Coker\phi\in R{\text-}\GProj$, it follows from Lemma \ref{lem:3.3} that $\left(
                                                                                        \begin{array}{c}
                                                                                          M\otimes_SY \\
                                                                                          Y \\
                                                                                        \end{array}
                                                                                      \right)$ and $\left(
                                                                                        \begin{array}{c}
                                                                                          \Coker\phi \\
                                                                                          0 \\
                                                                                        \end{array}
                                                                                      \right)$ are Gorenstein projective left $T$-modules.
Notice that $T{\text-}\GProj$ is closed under extensions (see \cite[Theorem 2.5]{Ho}), one has $\left(
                                                                                        \begin{array}{c}
                                                                                          X \\
                                                                                          Y \\
                                                                                        \end{array}
                                                                                      \right)_\phi\in T{\text-}\GProj$.

Conversely, assume $\left(
                                                                                        \begin{array}{c}
                                                                                          X \\
                                                                                          Y \\
                                                                                        \end{array}
                                                                                      \right)_\phi\in T{\text-}\GProj$. Then there exists a complete projective $T$-complex
$$L^\bullet=\cdots\to \left(\begin{array}{c}P_1^{-1} \\P_2^{-1} \\\end{array}\right)_{\phi^{-1}}\xrightarrow{\left(\begin{array}{c}d_1^{-1} \\d_2^{-1} \\\end{array}\right)}\left(\begin{array}{c}P_1^{0} \\P_2^{0} \\\end{array}\right)_{\phi^{0}}\xrightarrow{\left(\begin{array}{c}d_1^{0} \\d_2^{0} \\\end{array}\right)}\left(\begin{array}{c}P_1^{1} \\P_2^{1} \\\end{array}\right)_{\phi^{1}}\to\cdots$$ with $\left(
                                                                                        \begin{array}{c}
                                                                                          X \\
                                                                                          Y \\
                                                                                        \end{array}
                                                                                      \right)_\phi\cong Z^0(L^\bullet)$. Denote by $P_1^\bullet=\cdots\to P_1^{-1}\xrightarrow{d_1^{-1}}P_1^{0}\xrightarrow{d_1^{0}}P_1^{1}\to\cdots$ and $P_2^\bullet=\cdots\to P_2^{-1}\xrightarrow{d_2^{-1}}P_2^{0}\xrightarrow{d_2^{0}}P_2^{1}\to\cdots$. Then $P_1^\bullet$ is an acyclic complex of left $R$-modules with $X\cong Z^0(P_1^\bullet)$ and $P_2^\bullet$ is an acyclic complex of projective left $S$-modules with $Y\cong Z^0(P_2^\bullet)$. Now let $Q\in S{\text-}\Proj$, by Lemma \ref{lem:2.3}, $\pd_T\left(\begin{array}{c}0 \\Q \\\end{array}\right)<\infty$. It follows that $\Hom_S(P_2^\bullet,Q)\cong\Hom_T(L^\bullet,\left(\begin{array}{c}0 \\Q \\\end{array}\right))$ is acyclic. Then $P_2^\bullet$ is a complete projective $S$-complex and hence $Y\in S{\text-}\GProj$.

Next, we will check $\phi:M\otimes_SY\to X$ is injective. Since $\fd M_S<\infty$, $M\otimes_SP_2^\bullet$ is acyclic by Lemma \ref{lem:2.1}. It follows that $M\otimes_SY\cong Z^0(M\otimes_SP_2^\bullet)$ and then $\phi:M\otimes_SY\to X$ is the restriction of $\phi^0:M\otimes_SP_2^0\to P_1^0$. Notice that $\phi^0:M\otimes_SP_2^0\to P_1^0$ is injective, thus $\phi:M\otimes_SY\to X$ is injective.

Finally, we prove $\Coker\phi\in R{\text-}\GProj$. Denote by $\widetilde{P_1^\bullet}=\cdots\to\Coker\phi^{-1}\to\Coker\phi^{0}\to\Coker\phi^{1}\to\cdots$ with canonical differentials. Then we get the following exact sequence of acyclic $R$-complexes:
$$0\to M\otimes P_2^\bullet\to P_1^\bullet\to \widetilde{P_1^\bullet}\to 0.$$
Since $M\otimes P_2^\bullet$ and $P_1^\bullet$ are acyclic, one has $\widetilde{P_1^\bullet}$ is acyclic. It follows that $0\to Z^0(M\otimes P_2^\bullet)\to Z^0(P_1^\bullet)\to Z^0(\widetilde{P_1^\bullet})\to0$ is exact. Notice that $M\otimes_SY\cong Z^0(M\otimes_SP_2^\bullet)$ and $X\cong Z^0(P_1^\bullet)$, we conclude $\Coker\phi\cong Z^0(\widetilde{P_1^\bullet})$.
To complete this assertion, it suffices to show $\widetilde{P_1^\bullet}$ is complete projective.
Since each $\left(\begin{array}{c}P_1^{i} \\P_2^{i} \\\end{array}\right)_{\phi^{i}}\in T{\text-}\Proj$, it follows that each $\Coker\phi^i\in R{\text-}\Proj$. Thus $\widetilde{P_1^\bullet}$ is an acyclic complex of projective $R$-modules. For any $P\in R{\text-}\Proj$, one has $\left(
                                                                                                                                                                                                \begin{array}{c}
                                                                                                                                                                                                  P \\
                                                                                                                                                                                                  0 \\
                                                                                                                                                                                                \end{array}
                                                                                                                                                                                              \right)\in T{\text-}\Proj$. Then we get
$\Hom_R(\widetilde{P_1^\bullet},P)\cong\Hom_T(L^\bullet,\left(
                                                                                                                                                                                                \begin{array}{c}
                                                                                                                                                                                                  P \\
                                                                                                                                                                                                  0 \\
                                                                                                                                                                                                \end{array}
                                                                                                                                                                                              \right)
)$ is acyclic. Therefore, $\widetilde{P_1^\bullet}$ is complete projective.\qed

%
%\begin{rem} {\rm In \cite{ZP}, Zhang described finitely generated left Gorenstein projective $T$-modules in case that $T$ is an Artin algebra. When $T$ is an arbitrary triangular matrix ring with $R$ left Gorenstein regular, Gorenstein projective modules over $T$ are studied by Enochs, Cort¨¦s-Izurdiaga and Torrecillas (see \cite[Theorem3.5]{ECT}). In this paper, we get the same result as in \cite{ECT} for any triangular matrix ring $T$.
%}\end{rem}

The following assertion could be checked directly.
\begin{prop}\label{prop:5.8}  Let $T=\left(
                                                                                 \begin{array}{cc}
                                                                                   R & M \\
                                                                                   0 & S \\
                                                                                 \end{array}
                                                                               \right)
$ be a triangular matrix ring with $\pd_RM<\infty$ and $\fd M_S<\infty$. Then
$\Gpd {_T}\left(
                               \begin{array}{c}
                                 X \\
                                 0 \\
                               \end{array}
                             \right)
=\Gpd {_R}X$, for any $X\in R{\text-}\Mod$.
\end{prop}

Let $A$ be any ring, recall from \cite{ECT} that $A$ is {\it strongly left CM-free} if every left Gorenstein projective $A$-module is projective.

\begin{cor}\label{cor:5.9}(compare \cite[Theorem 4.1]{ECT}) Let $T=\left(
                                                                                 \begin{array}{cc}
                                                                                   R & M \\
                                                                                   0 & S \\
                                                                                 \end{array}
                                                                               \right)
$ be a triangular matrix ring with $\pd_RM<\infty$ and $\fd M_S<\infty$. Then $T$ is strongly left CM-free if and only if $R$ and $S$ are strongly left CM-free.
\end{cor}
\begin{proof} For the ``if'' part, assume $R$ and $S$ be strongly left CM-free. Let $\left(\begin{array}{c}X \\Y \\\end{array}\right)_{\phi}\in T{\text-}\GProj$. It follows from Theorem \ref{thm:1.1} that $Y\in S{\text-}\GProj$ and $\phi:M\otimes_S Y\to X$ is an injective $R$-morphism with $\Coker\phi\in R{\text-}\GProj$. Since $R$ and $S$ are strongly left CM-free, $Y\in S{\text-}\Proj$ and $\Coker\phi\in R{\text-}\Proj$. By Lemma \ref{lem:2.2}, $\left(\begin{array}{c}X \\Y \\\end{array}\right)_{\phi}\in T{\text-}\Proj$ and then $T$ is strongly left CM-free.

Conversely, assume $T$ is strongly left CM-free. Let $X\in R{\text-}\GProj$ and $Y\in S{\text-}\GProj$. It follows from Lemma \ref{lem:3.3} that  $\left(
                                                                                        \begin{array}{c}
                                                                                          M\otimes_SY \\
                                                                                          Y \\
                                                                                        \end{array}
                                                                                      \right), \left(\begin{array}{c}
                                                                                          X \\
                                                                                          0 \\
                                                                                        \end{array}
                                                                                      \right)\in T{\text-}\GProj$ and then $\left(
                                                                                        \begin{array}{c}
                                                                                          M\otimes_SY \\
                                                                                          Y \\
                                                                                        \end{array}
                                                                                      \right), \left(\begin{array}{c}
                                                                                          X \\
                                                                                          0 \\
                                                                                        \end{array}
                                                                                      \right)\in T{\text-}\Proj$.                                                                                      By Lemma \ref{lem:2.2}, $X\in R{\text-}\Proj$ and $Y\in S{\text-}\Proj$. Hence $R$ and $S$ are strongly left CM-free.
\end{proof}

\section{\bf Recollements over triangular matrix rings}\label{singcat}

Let $R$ and $S$ be two rings, ${_R}M_S$ be an $R$-$S$-bimodule, and $T=\left(
                                                                                 \begin{array}{cc}
                                                                                   R & M \\
                                                                                   0 & S \\
                                                                                 \end{array}
                                                                               \right)
$ be the triangular matrix ring. It is shown in section \ref{Gdc} that, $\D^b( T{\text-}\Mod)_{fgp}$ is a thick subcategory of $\D^b( T{\text-}\Mod)$ and it plays an important role in the singularity category and Gorenstein defect category. In this section, we will study (left) recollements for $T$ corresponding to $\D^b( T{\text-}\Mod)_{fgp}$.
Parts of the work in this section are inspired by Lu \cite{L}.

To take $A=T$ and $e=e_S=\left(
                           \begin{array}{cc}
                             0 & 0 \\
                             0 & 1 \\
                           \end{array}
                         \right)$ as in Lemma \ref{lem:2.6}, we get the following. See also \cite{CL,L,ZP} for example.

\begin{lem}\label{lem:5.1}  Let $R$ and $S$ be two rings, ${_R}M_S$ an $R$-$S$-bimodule, and $T=\left(
                                                                                 \begin{array}{cc}
                                                                                   R & M \\
                                                                                   0 & S \\
                                                                                 \end{array}
                                                                               \right)
$ the triangular matrix ring.
\begin{enumerate}
\item[(1)] We have the following recollement of module categories:
$$\xymatrix{ R{\text-}\Mod\ar[r]^{i_*}& T{\text-}\Mod\ar@/^1pc/[l]^{i^!}\ar@/_1pc/[l]_{i^*}\ar[r]^{j^*} & S{\text-}\Mod\ar@/^1pc/[l]^{j_*}\ar@/_1pc/[l]_{j_!},}$$
where $i^*=R\otimes_T-$ is given by $\left(\begin{array}{c} X \\Y \\\end{array}\right)_{\phi}\mapsto\Coker\phi$; $i_*$ is given by $X\mapsto\left(\begin{array}{c} X \\0 \\\end{array}\right)$;
$i^!=\Hom_T(R,-)$ is given by $\left(\begin{array}{c} X \\Y \\\end{array}\right)_{\phi}\mapsto X$; $j_!=Te_S\otimes_{S}-$ is given by $Y\mapsto \left(\begin{array}{c} M\otimes_SY \\Y \\\end{array}\right)$;
$j^*=e_ST\otimes_T-$ is given by $\left(\begin{array}{c} X \\Y \\\end{array}\right)_{\phi}\mapsto Y$ and $j_*=\Hom_{S}(e_ST,-)$ is given by $Y\mapsto\left(\begin{array}{c} 0 \\Y \\\end{array}\right)$.
\item[(2)] If $\pd M_s<\infty$, we have the following recollement of bounded derived categories:
$$\xymatrix{\D^b( R{\text-}\Mod)\ar[r]^{\D^b(i_*)}&\D^b( T{\text-}\Mod)\ar@/^1pc/[l]^{\D^b(i^!)}\ar@/_1pc/[l]_{\mathbb{L}^b(i^*)}\ar[r]^{\D^b(j^*)} &\D^b( S{\text-}\Mod)\ar@/^1pc/[l]^{\D^b(j_*)}\ar@/_1pc/[l]_{\mathbb{L}^b(j_!)}, }\eqno{(1.1)}$$
where these six functors are the derived versions of those in (1).
\end{enumerate}
\end{lem}

Restrict (1.1) to the bounded homotopy categories of projective modules, we get the following.

\begin{lem}\label{lem:5.2} Let $T=\left(
                                                                                 \begin{array}{cc}
                                                                                   R & M \\
                                                                                   0 & S \\
                                                                                 \end{array}
                                                                               \right)
$ be a triangular matrix ring with $\pd_RM<\infty$ and $\pd M_S<\infty$. Then (1.1) restricts the following recollement:
$$\xymatrix{\K^b( R{\text-}\Proj)\ar[r]^{\D^b(i_*)}&\K^b( T{\text-}\Proj)\ar@/^1pc/[l]^{\D^b(i^!)}\ar@/_1pc/[l]_{\mathbb{L}^b(i^*)}\ar[r]^{\D^b(j^*)} &\K^b( S{\text-}\Proj)\ar@/^1pc/[l]^{\D^b(j_*)}\ar@/_1pc/[l]_{\mathbb{L}^b(j_!)}}.\eqno{(1.2)}$$
Furthermore, (1.1) induces the following recollement:
$$\xymatrix{\D_{sg}( R{\text-}\Mod)\ar[r]&\D_{sg}( T{\text-}\Mod)\ar@/^1pc/[l]\ar@/_1pc/[l]\ar[r] &\D_{sg}( S{\text-}\Mod)\ar@/^1pc/[l]\ar@/_1pc/[l]}.$$
\end{lem}
\begin{proof} See \cite[Theorem 3.2]{LL}.
\end{proof}
%\begin{prop}\label{prop:5.3}(\cite[Theorem 3.2]{LL}) Let $T=\left(
%                                                                                 \begin{array}{cc}
%                                                                                   R & M \\
%                                                                                   0 & S \\
%                                                                                 \end{array}
%                                                                               \right)
%$ be a triangular matrix ring with $\pd_RM<\infty$ and $\pd M_S<\infty$. Then (1.1) induces the following recollement:
%$$\xymatrix{\D_{sg}( R{\text-}\Mod)\ar[r]&\D_{sg}( T{\text-}\Mod)\ar@/^1pc/[l]\ar@/_1pc/[l]\ar[r] &\D_{sg}( S{\text-}\Mod)\ar@/^1pc/[l]\ar@/_1pc/[l]}.$$
%\end{prop}

{\bf Proof of Theorem \ref{thm:1.3} (1).} We first claim $\D^b(i_*)(\D^b( R{\text-}\Mod)_{fgp})\subseteq\D^b(T{\text-}\Mod )_{fgp}$. In fact, let $G\in R{\text-}\GProj$. It follows from Lemma \ref{lem:5.1} that $i_*(G)\cong\left(\begin{array}{c} G\\0 \\\end{array}\right)$ and then $i_*(G)\in T{\text-}\GProj$ by Lemma \ref{lem:3.3}. Now let $X^\bullet\in\D^b( R{\text-}\Mod)_{fgp}$, it follows that $X^\bullet$ is isomorphic to some bounded complex $G^\bullet$ consistig of Gorenstein projective $R$-modules. Then $\D^b(i_*)(X^\bullet)\cong\D^b(i_*)(G^\bullet)\cong i_*(G^\bullet)$ is a bounded complex consistig of Gorenstein projective $T$-modules. Hence $\D^b(i_*)(\D^b( R{\text-}\Mod)_{fgp})\subseteq\D^b( T{\text-}\Mod)_{fgp}$.

Next we show $\mathbb{L}^b(i^*)(\D^b( T{\text-}\Mod)_{fgp})=\D^b( R{\text-}\Mod)_{fgp}$. Since $i^*(\left(\begin{array}{c} X \\Y \\\end{array}\right)_{\phi})\cong\Coker\phi$, one has $i^*$ preserves Gorenstein projective modules by Theorem \ref{thm:1.1}.
 Now let $X^\bullet\in\D^b( T{\text-}\Mod)_{fgp}$, it follows from Lemma \ref{lem:A.2} that there exists a quasi-isomorphism $P^\bullet\to X^\bullet$ with $P^\bullet\in\K^{-,b}(T{\text-}\Proj)$ such that $Z^{i}(P^\bullet)\in T{\text-}\GProj$ for $i\ll0$. Denote by $s$ the index such that $H^i(P^\bullet)=0$ and $Z^{i}(P^\bullet)\in T{\text-}\GProj$ for any $i\leq s$.
 It follows that $\cdots\to P^{s-2}\to P^{s-1}\to Z^s(P^\bullet)\to0$ is a projective resolution of $Z^s(P^\bullet)$. Notice that $Z^{s}(P^\bullet)\in T{\text-}\GProj$, we get the following acyclic $T$-complex with each degree projective $$P'^\bullet=\cdots\to P^{s-2}\to P^{s-1}\to Q^s\to Q^{s+1}\to\cdots.$$
Since $\pd M_S<\infty$, by Lemma \ref{lem:2.5} $\pd R_T<\infty$. It follows from Lemma \ref{lem:2.1} that $i^*(P'^\bullet)\cong R\otimes_TP'^\bullet$ is acyclic, and then $\cdots\to R\otimes_TP^{s-2}\to R\otimes_TP^{s-1}\to R\otimes_TZ^s(P^\bullet)\to0$ is acyclic. Therefore, $i^*(P^\bullet)$ and hence $\mathbb{L}^b(i^*)(X^\bullet)$ are isomorphic to the complex $0\to R\otimes_TZ^s(P^\bullet)\to R\otimes_T P^s\to R\otimes_T P^{s+1}\to\cdots$, which is a bounded complex with each degree Gorenstein projective.
Hence $\mathbb{L}^b(i^*)(X^\bullet)\in\D^b( R{\text-}\Mod)_{fgp}$ and therefore $\mathbb{L}^b(i^*)(\D^b( T{\text-}\Mod)_{fgp})\subseteq\D^b( R{\text-}\Mod)_{fgp}$.
On the other hand, let $Y^\bullet\in\D^b( R{\text-}\Mod)_{fgp}$, then $Y^\bullet\cong \mathbb{L}^b(i^*)(\D^b(i_*)(Y^\bullet))$. By the above claim, $\D^b(i_*)(Y^\bullet)\in\D^b( T{\text-}\Mod)_{fgp}$. Then $Y^\bullet\in \mathbb{L}^b(i^*)(\D^b( T{\text-}\Mod)_{fgp})$ and hence $\D^b( R{\text-}\Mod)_{fgp}\subseteq \mathbb{L}^b(i^*)(\D^b( T{\text-}\Mod)_{fgp})$. Therefore $\mathbb{L}^b(i^*)(\D^b( T{\text-}\Mod)_{fgp})$ $=\D^b( R{\text-}\Mod)_{fgp}$.

Now note that both $j_!$ and $j^*$ preserve Gorenstein projective modules. Actually, by the similar arguments as above, one gets $\mathbb{L}^b(j_!)(\D^b( S{\text-}\Mod)_{fgp})\subseteq\D^b( T{\text-}\Mod)_{fgp}$ and $\D^b(j^*)(\D^b( T{\text-}\Mod)_{fgp})=\D^b( S{\text-}\Mod)_{fgp}$. Therefore, (1.1) restricts to the following left recollement:
$$\xymatrix{\D^b( R{\text-}\Mod)_{fgp}\ar[r]^{\D^b(i_*)}&\D^b( T{\text-}\Mod)_{fgp}\ar@/_1pc/[l]_{\mathbb{L}^b(i^*)}\ar[r]^{\D^b(j^*)} &\D^b( S{\text-}\Mod)_{fgp}\ar@/_1pc/[l]_{\mathbb{L}^b(j_!)}}.\eqno{(1.3)}$$
\qed

\begin{cor}\label{cor:5.3} (compare \cite[Proposition 3.6]{L} and \cite[Theorem 3.3]{ZP})  Let $T=\left(
                                                                                 \begin{array}{cc}
                                                                                   R & M \\
                                                                                   0 & S \\
                                                                                 \end{array}
                                                                               \right)
$ be a triangular matrix ring with $\pd_RM<\infty$ and $\pd M_S<\infty$. Then we have the following left recollement of stable categories:
$$\xymatrix{\underline{R{\text-}\GProj}\ar[r]&\underline{T{\text-}\GProj}\ar@/_1pc/[l]\ar[r] &\underline{S{\text-}\GProj}\ar@/_1pc/[l]},$$
and the following left recollement of Gorenstein defect categories:
$$\xymatrix{\D_{def}( R{\text-}\Mod)\ar[r]&\D_{def}( T{\text-}\Mod)\ar@/_1pc/[l]\ar[r] &\D_{def}( S{\text-}\Mod)\ar@/_1pc/[l]}.$$
\end{cor}
\begin{proof} Combine recollements (1.2) and (1.3) with \cite[Lemma 2.3]{L}, we have the following left recollement:

$$\xymatrix{\D^b( R{\text-}\Mod)_{fgp}/\K^b( R{\text-}\Proj)\ar[r]&\D^b( T{\text-}\Mod)_{fgp}/\K^b( T{\text-}\Proj)\ar@/_1pc/[l]\ar[r] &\D^b( S{\text-}\Mod)_{fgp}/\K^b( S{\text-}\Proj)\ar@/_1pc/[l]}.$$
By Theorem \ref{thm:A.5}, we obtain the following recollement:
$$\xymatrix{\underline{R{\text-}\GProj}\ar[r]&\underline{T{\text-}\GProj}\ar@/_1pc/[l]\ar[r] &\underline{S{\text-}\GProj}\ar@/_1pc/[l]}.$$
Similarly, one could get the following left recollement
$$\xymatrix{\D_{def}( R{\text-}\Mod)\ar[r]&\D_{def}( T{\text-}\Mod)\ar@/_1pc/[l]\ar[r] &\D_{def}( S{\text-}\Mod)\ar@/_1pc/[l]}$$
by recollements (1.1) and (1.3), \cite[Lemma 2.3]{L} and Theorem \ref{thm:A.5}.
\end{proof}

{\bf Proof of Theorem \ref{thm:1.3} (2).} It remains to show $\D^b(j_*)(\D^b( S{\text-}\Mod)_{fgp})\subseteq\D^b( T{\text-}\Mod)_{fgp}$ and $\D^b(i^!)(\D^b( T{\text-}\Mod)_{fgp})\subseteq\D^b( R{\text-}\Mod)_{fgp}$.
Let $G\in S{\text-}\GProj$, we get $j_*(G)=\left(
                                                                                                            \begin{array}{c}
                                                                                                              0 \\
                                                                                                              G \\
                                                                                                            \end{array}
                                                                                                          \right)$.
Now consider the following exact sequence of left $T$-modules
$$0\to\left(
        \begin{array}{c}
          M\otimes_S G \\
          0 \\
        \end{array}
      \right)\to\left(
                  \begin{array}{c}
                    M\otimes_S G \\
                    G \\
                  \end{array}
                \right)\to\left(
                                 \begin{array}{c}
                                   0 \\
                                   G \\
                                 \end{array}
                               \right)\to0.$$
Since $G\in S{\text-}\GProj$, $\left(
                  \begin{array}{c}
                    M\otimes_S G \\
                    G \\
                  \end{array}
                \right)\in T{\text-}\GProj$ by Lemma \ref{lem:3.3}. It follows that $\Gpd_T\left(
                                 \begin{array}{c}
                                   0 \\
                                   G \\
                                 \end{array}
                               \right)<\infty$ if and only if $\Gpd_T\left(
        \begin{array}{c}
          M\otimes_S G \\
          0 \\
        \end{array}
      \right)<\infty$.
By assumption, $\Gpd {_R}M\otimes_S G<\infty$. It follows from Proposition \ref{prop:5.8} that $\Gpd_T\left(
        \begin{array}{c}
          M\otimes_S G \\
          0 \\
        \end{array}
      \right)<\infty$.
Then $j_*(G)$ is of finite Gorenstein projective dimension as a left $T$-module, for any $G\in S{\text-}\GProj$.
Now let $X^\bullet\in\D^b( S{\text-}\Mod)_{fgp}$, then $X^\bullet\cong G^\bullet$
for some $G^\bullet\in\C^b( S{\text-}\GProj)$. It follows that $\D^b(j_*)(X^\bullet)\cong j_*(G^\bullet)$ is a bounded complex with each degree being of finite Gorenstein projective dimension. It follows from Lemma \ref{prop:A.4} that $\D^b(j_*)(X^\bullet)\in\D^b( T{\text-}\Mod)_{fgp}$ and hence $\D^b(j_*)(\D^b( S{\text-}\Mod)_{fgp})\subseteq\D^b( T{\text-}\Mod)_{fgp}$.

On the other hand, let $X^\bullet\in\D^b( T{\text-}\Mod)_{fgp}$. Then we have the following triangle in $\D^b( T{\text-}\Mod)$
$$\D^b(i_*)\D^b(i^!)X^\bullet\to X^\bullet\to \D^b(j_*)\D^b(j^*)X^\bullet\to \D^b(i_*)\D^b(i^!)X^\bullet[1].$$
Since $\D^b(j_*)\D^b(j^*)X^\bullet\in\D^b( T{\text-}\Mod)_{fgp}$ and $\D^b( T{\text-}\Mod)_{fgp}$ is a thick subcategory of $\D^b( T{\text-}\Mod)$, $\D^b(i_*)\D^b(i^!)X^\bullet\in\D^b( T{\text-}\Mod)_{fgp}$. One has
$\D^b(i^!)X^\bullet\cong \mathbb{L}^b(i^*)(\D^b(i_*)\D^b(i^!)X^\bullet)$, and then $\D^b(i^!)X^\bullet\in\D^b( R{\text-}\Mod)_{fgp}$ by the proof of Theorem \ref{thm:1.3} (1).  Therefore, $\D^b(i^!)(\D^b( T{\text-}\Mod)_{fgp})\subseteq\D^b( R{\text-}\Mod)_{fgp}$.
\qed

\begin{cor}\label{cor:5.4} Let $T=\left(
                                                                                 \begin{array}{cc}
                                                                                   R & M \\
                                                                                   0 & S \\
                                                                                 \end{array}
                                                                               \right)
$ be a triangular matrix ring with $\pd_RM<\infty$ and $\pd M_S<\infty$. If $\Gpd {_R}M\otimes_S G<\infty$ for any $G\in S{\text-}\GProj$. Then we have the following recollement of stable categories:

$$\xymatrix{\underline{R{\text-}\GProj}\ar[r]&\underline{T{\text-}\GProj}\ar@/^1pc/[l]\ar@/_1pc/[l]\ar[r] &\underline{S{\text-}\GProj}\ar@/^1pc/[l]\ar@/_1pc/[l] },$$
\noindent and the following recollement of Gorenstein defect categories:

$$\xymatrix{\D_{def}( R{\text-}\Mod)\ar[r]&\D_{def}( T{\text-}\Mod)\ar@/^1pc/[l]\ar@/_1pc/[l]\ar[r] &\D_{def}( S{\text-}\Mod)\ar@/^1pc/[l]\ar@/_1pc/[l] }.$$
\end{cor}
\begin{proof} Combine the recollements (1.2) and (1.4) with \cite[Proposition 2.5]{LL}, we have the following recollement:
$$\xymatrix{\D^b( R{\text-}\Mod)_{fgp}/\K^b( R{\text-}\Proj)\ar[r]&\D^b( T{\text-}\Mod)_{fgp}/\K^b( T{\text-}\Proj)\ar@/^1pc/[l]\ar@/_1pc/[l]\ar[r] &\D^b( S{\text-}\Mod)_{fgp}/\K^b( S{\text-}\Proj)\ar@/^1pc/[l]\ar@/_1pc/[l] }.$$
Then by Theorem \ref{thm:A.5} we get the following recollement: $$\xymatrix{\underline{R{\text-}\GProj}\ar[r]&\underline{T{\text-}\GProj}\ar@/^1pc/[l]\ar@/_1pc/[l]\ar[r] &\underline{S{\text-}\GProj}\ar@/^1pc/[l]\ar@/_1pc/[l] } .$$
Similarly, one could get the following recollement
$$\xymatrix{\D_{def}( R{\text-}\Mod)\ar[r]&\D_{def}( T{\text-}\Mod)\ar@/^1pc/[l]\ar@/_1pc/[l]\ar[r] &\D_{def}( S{\text-}\Mod)\ar@/^1pc/[l]\ar@/_1pc/[l]}$$
 by recollements (1.1) and (1.4), \cite[Proposition 2.5]{LL} and Theorem \ref{thm:A.5}.
\end{proof}

 Let $T=\left(
                                                                                 \begin{array}{cc}
                                                                                   R & M \\
                                                                                   0 & S \\
                                                                                 \end{array}
                                                                               \right)
$ be a triangular matrix ring with $\pd_RM<\infty$ and $\pd M_S<\infty$. In the following , we will list some examples to indicate that the condition ``$\Gpd {_R}M\otimes_S G<\infty$ for any $G\in S{\text-}\GProj$'' in Theorem \ref{thm:1.3} arises naturally in representation theory.

\begin{exa}\label{exa:5.5} {\rm\begin{enumerate}
\item[(1)] Let $R$ be left Gorenstein regular in the sense of \cite{ECT,EEG}. Then it follows from \cite[Theorem 2.28]{EEG} that every left $R$-module has finite Gorenstein projective dimension and then $\Gpd {_R}M\otimes_S G<\infty$ for any $G\in S{\text-}\GProj$. In this case, we get a triangle-equivalence $\D_{def}( T{\text-}\Mod)\simeq\D_{def}( S{\text-}\Mod)$ by Corollary \ref{cor:A.6}.
\item[(2)] Let $S$ be strongly left CM-free. Then every Gorenstein projective $S$-module is projective and then $\pd {_R}M\otimes_S G<\infty$ for any $G\in S{\text-}\GProj$. In this case, we get a triangle-equivalence $\underline{R{\text-}\GProj}\simeq\underline{T{\text-}\GProj}$.
\item[(3)] (compare \cite[Theorem 3.12]{L}) Let $M\in R{\text-}\Proj$ and $\Hom_R(M,P)\in( S{\text-}\GProj)^\bot$ for any $P\in R{\text-}\Proj$. Then $M\otimes_S G\in R{\text-}\GProj$ for any $G\in S{\text-}\GProj$.
Indeed, let $G\in S{\text-}\GProj$. Then there exist a complete projective $S$-complex $T^\bullet$ such that $G\cong Z^0(T^\bullet)$. Notice that $_RM$ is projective and $\pd M_S<\infty$, $M\otimes_S T^\bullet$ is an acyclic complex of projective $R$-modules with $M\otimes_SG\cong Z^0(M\otimes_ST^\bullet)$. For any $P\in R{\text-}\Proj$, one has $\Hom_R(M\otimes_S T^\bullet,P)\cong\Hom_S(T^\bullet,\Hom_R(M,P))$ is acyclic. Thus $M\otimes_S G\in R{\text-}\GProj$.
\item[(4)] (compare \cite[Corollary 3.13]{L}) Let $M\in R{\text-}\Proj$ and $M\otimes_S G\in{^\bot}R{\text-}\Proj$ for any $G\in S{\text-}\GProj$. Then $M\otimes_S G\in R{\text-}\GProj$ for any $G\in S{\text-}\GProj$. In fact, let $G\in S{\text-}\GProj$. Then there exist a complete projective $S$-complex $T^\bullet$ such that $G\cong Z^0(T^\bullet)$. By assumption, $M\otimes_S T^\bullet$ is an acyclic complex of projective $R$-modules with $M\otimes_SG\cong Z^0(M\otimes_ST^\bullet)$. Since $M\otimes_S G\in{^\bot}R{\text-}\Proj$ for any $G\in S{\text-}\GProj$, $M\otimes_S T^\bullet$ is complete projective and then $M\otimes_S G\in R{\text-}\GProj$.
\end{enumerate}}
\end{exa}

\appendix
 \renewcommand{\appendixname}{Appendix~\Alph{section}}
  \section{Triangle-equivalences associative to Gorenstein projective dimension for complexes}\label{Gdc}

In this appendix, $A$ is assumed to be an arbitrary ring.  Denote by $\K^*(\p)$ the $*$-bounded homotopy category of $\p$, where $*\in\{+,\ -, \ b\}$.
For any subcategory $\X\subseteq\A$, denote by $\langle \X\rangle$ the smallest triangulated subcategory of $\D^b(\A)$ containing $\X$. For example $\langle \GP\rangle$ is the smallest triangulated subcategory of $\D^b(\A)$.

As is known, the comparison of $\K^b(A{\text-}\Proj)$ with $\D^b(A{\text-}\Mod)$ reflecs certain homological singularity of the ring $A$ in sense that $\K^b(A{\text-}\Proj)=\D^b(A{\text-}\Mod)$ if and only if every $A$-module has finite projective dimension. Besides, the Verdier quotient $\D_{sg}(A{\text-}\Mod):=\D^b(A{\text-}\Mod)/\K^b(A{\text-}\Proj)$ was studied by Buchweitz \cite{Bu} and Orlov \cite{O} under the name of ``singularity category''.
Denote by $\underline{\GP}$ the stable category of Gorenstein projective $A$-modules. Buchweitz's Theorem (\cite[Theorem 4.4.1]{Bu}) says that there is a fully faithful triangle functor $F:\underline{\GP}\to\D_{sg}(\A)$, and $F$ is a triangle-equivalence provided that every $A$-module has finite Gorenstein projective dimension. Following this, Bergh, J{\o}rgensen and Oppermann (\cite{BJO}) considered the Verdier quotient $\D_{def}(\A):=\D_{sg}(\A)/\Im F$, and they called it the {\it Gorenstein defect category} of $\A$.
Nowadays, singularity category and related topic has been studied by many authors, see for example \cite{C1,C2,CZ,KZ,LL,L,R1,ZP}.

%Clearly, $\K^b(\p)=\langle \p\rangle$ and then $\K^b(\p)\subseteq\langle \GP\rangle\subseteq\D^b(\A)$. We wonder what the real forms of objects in $\langle \GP\rangle$.  %The following is well-known, see for e.g. \cite[Corollary 6.2]{KZ}.

An object $X^\bullet\in\D^b(\A)$ is said to have {\it finite Gorenstein projective dimension} if $X^\bullet$ is isomorphic to some bounded complex consisting of Gorenstein projective modules. Denote by $\D^b(\A)_{fgp}$
the subcategory of $\D^b(\A)$ consisting of complexes with finite Gorenstein projective dimension. This definition coincides with that of \cite{Vel} and \cite{ZH}.

\begin{rem} Let $M\in\A$. $M$ has finite Gorenstein projective dimension as an $A$-module if and only if $M$ has finite Gorenstein projective dimension as a stalk complex concentrated in degree zero.
\end{rem}

Clearly, $\K^b(\p)\subseteq\D^b(\A)_{fgp}\subseteq\D^b(\A)$. We wonder how $\D^b(\A)_{fgp}$ behaves in $\D^b(\A)$.

\begin{lem}\label{lem:A.2} Let $X^\bullet\in\D^b(\A)$. The following are equivalent
\begin{enumerate}
\item[(1)] $X^\bullet\in\D^b(\A)_{fgp}$.
\item[(2)]For any quasi-isomorphism $P^\bullet\to X^\bullet$ with $P^\bullet\in\K^{-,b}({\p})$, one has $Z^{i}(P^\bullet)\in\GP$ for $i\ll0$, where $\K^{-,b}(\p)$ the full subcategory of $\K^{-}(\p)$ consisting of complexes with finite nonzero homology.
\item[(3)] There exists a quasi-isomorphism $P^\bullet\to X^\bullet$ with $P^\bullet\in\K^{-,b}({\p})$ such that $Z^{i}(P^\bullet)\in\GP$ for $i\ll0$.
\end{enumerate}
\end{lem}
\begin{proof} (1)$\Longrightarrow$(2). Let $P^\bullet\to X^\bullet$ be a quasi-isomorphism with $P^\bullet\in\K^{-,b}({\p})$. Since $X^\bullet\in\D^b(\A)_{fgp}$,  $X^\bullet\cong G^\bullet$ in $\D(\A)$ for some $G^\bullet\in\C^b(\GP)$. Then $P^\bullet\cong G^\bullet$ in $\D^b(\A)$ and then there is a quasi-isomorphism $f:P^\bullet\to G^\bullet$ by \cite[1.4.P]{AFHd}. Hence $\Con(f)=\cdots\to P^{-n-1}\to P^{-n}\to P^{-n+1}\oplus G^{-n}\to\cdots$
is acyclic. Note that $\Con(f)$ is bounded above with each degree lies in $\GP$. It follows from \cite[Theorem 2.5]{Ho} that $Z^{i}(P^\bullet)\cong Z^{i-1}(\Con(f))\in\GP$ for $i\ll0$.

(2)$\Longrightarrow$(3) is trivial.

(3)$\Longrightarrow$(1). Let $P^\bullet\to X^\bullet$ be a quasi-isomorphism with $P^\bullet\in\K^{-,b}({\p})$ and $Z^{i}(P^\bullet)\in\GP$ for $i\ll0$. Since $P^\bullet\in\K^{-,b}({\p})$, $P^\bullet$ is isomorphic to the following complex in $\D^b(\A)$
$$G^\bullet:=0\to Z^{t}(P^\bullet)\to P^{t}\to\cdots\to P^{s-1}\to P^s\to0,$$
where $s$ is the supremum of index $i\in\mathbb{Z}$ such that $P^i\neq0$ and $t$ is the index such that $H^i(P^\bullet)=0$ for any $i\leq t$. Hence $X^\bullet\cong G^\bullet$ in $\D^b(\A)$ with $G^\bullet\in\C^b(\GP)$ and then $X^\bullet\in\D^b(\A)_{fgp}$.
\end{proof}

For each $X^\bullet\in\C(\A)$. The {\it length} $l(X^\bullet)$ of $X^\bullet$ is defined to be the cardinal of the set $\{X^i\neq0|i\in\mathbb{Z}\}$. Let $n\in\mathbb{Z}$, denote by $X^\bullet_{\geqslant n}$ the complex with the $i$th component equal to $X^i$ whenever $i\geqslant n$ and to 0 elsewhere.

\begin{thm}\label{thm:A.3} $\D^b(\A)_{fgp}$ is a thick subcategory of $\D^b(\A)$. Furthermore, $\D^b(\A)_{fgp}=\langle \GP\rangle$.
\end{thm}
\begin{proof}  It follows from \cite[Proposition 3.2]{ZH} that $\D^b(\A)_{fgp}$ is a triangulated subcategory of $\D^b(\A)$. To get the first assertion, it suffices to show $\D^b(\A)_{fgp}$ is closed under direct summands. In fact, let $X_1^\bullet\oplus X_2^\bullet\in\D^b(\A)_{fgp}$ with $X_1^\bullet, X_2^\bullet\in\D^b(\A)$. Choose quasi-isomorphisms $P_1^\bullet\to X_1^\bullet$ and $P_2^\bullet\to X_2^\bullet$ with $P_1^\bullet, P_2^\bullet\in\K^{-,b}({\p})$. It follows that $P_1^\bullet\oplus P_2^\bullet\to X_1^\bullet\oplus X_2^\bullet$ is a quasi-isomorphism. Notice that $X_1^\bullet\oplus X_2^\bullet\in\D^b(\A)_{fgp}$ and $P_1^\bullet\oplus P_2^\bullet\in\K^{-,b}({\p})$, it follows from Lemma \ref{lem:A.2} that $Z^{i}(P_1^\bullet\oplus P_2^\bullet)\in\GP$ for $i\ll0$. Since $Z^{i}(P_1^\bullet\oplus P_2^\bullet)\cong Z^{i}(P_1^\bullet)\oplus Z^{i}(P_2^\bullet)$, we get $Z^{i}(P_1^\bullet), Z^{i}(P_2^\bullet)\in\GP$ for $i\ll0$. Then by Lemma \ref{lem:A.2} $X_1^\bullet, X_2^\bullet\in\D^b(\A)_{fgp}$ and hence $\D^b(\A)_{fgp}$ is closed under direct summands.

Note that every Gorenstein projective module viewed as a stalk complex has finite Gorenstein projective dimension. Thus $\GP\subseteq\D^b(\A)_{fgp}$
and then $\langle \GP\rangle\subseteq\D^b(\A)_{fgp}$. On the other hand, let $X^\bullet\in\D^b(\A)_{fgp}$. Then $X^\bullet\cong G^\bullet$ in $\D^b(\A)$ for some $G^\bullet\in\C^b(\GP)$. We wil show $G^\bullet\in\langle \GP\rangle$ to complete the proof. We proceed by induction on the length $l(G^\bullet)$ of $G^\bullet$. If $l(G^\bullet)=1$, it is trivial to verify $G^\bullet\in\langle \GP\rangle$. Now let $l(G^\bullet)=n\geq2$. We may suppose $G^m\neq0$ and $G^i=0$ for $i<m$. The we have the following triangle in $\D^b(\A)$:
$$G^m[-m-1]\to G^\bullet_{\geq{m+1}}\to G^\bullet\to G^m[-m].$$
By the induction hypothesis, we have that both $G^m[-m-1]$ and $G^\bullet_{\geq{m+1}}$ lie in $\langle \GP\rangle$. Therefore $G^\bullet\in\langle \GP\rangle$.
\end{proof}

\begin{prop}\label{prop:A.4} Let $X^\bullet\in\D^b(\A)$. If each $X^i$ is of finite Gorenstein projective dimension as an $A$-module, then $X^\bullet\in\D^b(\A)_{fgp}$. Furthermore, $\D^b(\A)_{fgp}=\D^b(\A)$ if and only if every $A$-module has finite Gorenstein projective dimension.
\end{prop}

\begin{proof} We will proceed by induction on the length $l(X^\bullet)$ of $X^\bullet$. If $l(X^\bullet)=1$, it is trivial to verify $X^\bullet\in\D^b(\A)_{fgp}$. Now let $l(X^\bullet)=n\geq2$. We may suppose $X^m\neq0$ and $X^i=0$ for $i<m$. The we have the following triangle in $\D^b(\A)$:
$$X^m[-m-1]\to X^\bullet_{\geq{m+1}}\to X^\bullet\to G^m[-m].$$
By the induction hypothesis, we have that both $X^m[-m-1]$ and $X^\bullet_{\geq{m+1}}$ lie in $\D^b(\A)_{fgp}$. Therefore $X^\bullet\in\D^b(\A)_{fgp}$.
\end{proof}

%Recall from \cite{Bu} that the Verdier quotient $\D^b(\A)/\K^b(\p)$ is called the {\it singularity category} of $\A$ and denoted by $\D_{sg}(\A)$. Consider the composition of the embedding functor $\GP\hookrightarrow\D^b(\A)$ and the localization functor $\D^b(\A)\to\D_{sg}(\A)$. It induces a functor $F:\underline{\GP}\to\D_{sg}(\A)$, which sends every Gorenstein projective module to the stalk complex concentrated in degree zero.
%Following \cite{Bu}, $F$ is a fully faithful triangle functor. Inspired by this, Bergh, J{\o}rgensen and Oppermann \cite{BJO} consider the Verdier quotient $\D_{def}(\A):=\D_{sg}(\A)/\Im F$, and they call it the {\it Gorenstein defect category} of $\A$. Next, we show that $\D^b(\A)_{fgp}$ makes a great use in the singularity theory and we have the following (compare \cite{KZ} and \cite{ZH}).

\begin{thm}\label{thm:A.5} There exist two triangle-equivalences $\underline{\GP}\simeq\D^b(\A)_{fgp}/\K^b(\p)$ and  $\D_{def}(\A)\simeq\D^b(\A)/\D^b(\A)_{fgp}$.
\end{thm}
\begin{proof}  %By Buchweitz's Theorem (see \cite[Theorem 4.4.1]{Bu}), it suffices to show $\Im F=\D^b(\A)_{fgp}/\K^b(\p)$.
%
%Note that every Gorenstein projective object viewed as a complex concentrated in degree zero has finite Gorenstein projective dimension. Thus
%$$\Im F\subseteq\D^b(\A)_{fgp}/\K^b(\p).$$
%Now let $X^\bullet\in\D^b(\A)_{fgp}/\K^b(\p)$, it follows from Lemma \ref{lem:A.2} that there exists a quasi-isomorphism $P^\bullet\to X^\bullet$ with $P^\bullet\in\K^{-,b}({\p})$ such that $Z^{i}(P^\bullet)\in\GP$ for $i\ll0$.
%Hence $P^\bullet$ is isomorphic to the following complex in $\D^b(\A)$
%$$G^\bullet:=0\to Z^{t}(P^\bullet)\to P^{t}\to\cdots\to P^{s-1}\to P^s\to0,$$
%where $s$ is the supremum of index $i\in\mathbb{Z}$ such that $P^i\neq0$ and $t$ is the index such that $H^i(P^\bullet)=0$ for any $i\leq t$.
% We have the following triangle in $\D_{sg}(\A)$:
%$$C^{-n}(P^\bullet)[n-1]\to P^\bullet_{\geq-n+1}\to G^\bullet\to C^{-n}(P^\bullet)[n].$$
%Since $P^\bullet_{\geq-n+1}\in\K^b(\p)$, $G^\bullet\cong C^{-n}(P^\bullet)[n]$ and then $X^\bullet\cong C^{-n}(P^\bullet)[n]$ in $\D_{sg}(\A)$.
%Hence $X^\bullet[-n]\cong C^{-n}(P^\bullet)$, that is, $X^\bullet[-n]\in\Im F$. Since $\Im F$ is a triangulated subcategory, $X^\bullet\in\Im F$
%and then $\D^b(\A)_{fgp}/\K^b(\p)\subseteq\Im F.$
The first equivalence follows from \cite[Theorem 3.4]{ZH} and the second one follows from \cite[Lemma 6.1]{KZ}.
\end{proof}

Consequently, we get the following.

\begin{cor}\label{cor:A.6} The following are equivalent:
\begin{enumerate}
\item[(1)] $F:\underline{\GP}\to\D_{sg}(\A)$ is a triangle-equivalence;
\item[(2)] $\D_{def}(\A)=0$;
\item[(3)] Every $A$-module has finite Gorenstein projective dimension.
\end{enumerate}
\end{cor}
\begin{proof} (1)$\Leftrightarrow$(2) is trivial.

 (2)$\Rightarrow(3)$. Let $\D_{def}(\A)=0$. It follows from Theorem \ref{thm:A.5} that $\D^b(\A)=\D^b(\A)_{fgp}$, and hence every $A$-module has finite Gorenstein projective dimension by Proposition \ref{prop:A.4}.

 (3)$\Rightarrow$(2). Assume every $A$-module has finite Gorenstein projective dimension. It follows from Proposition \ref{prop:A.4} that $\D^b(\A)_{fgp}=\D^b(\A)$, and then $\D_{def}(\A)=0$ by Theorem \ref{thm:A.5}.
\end{proof}

\bigskip {\bf Acknowledgements}
\bigskip

This research was partially supported by NSFC (Grant No. 11501257, 11626179, 11671069, 11701455, 11771212), Qing Lan
Project of Jiangsu Province, Jiangsu Government Scholarship for Overseas Studies (JS-2019-328), Shaanxi Province Basic Research Program of Natural Science (Grant No. 2017JQ1012), Natural Science Foundation of Zhejiang
Provincial (LY18A010032) and Fundamental Research Funds for the Central Universities (Grant No. JB160703). The authors are grateful to the referee for the valuable comments.

\end{document}